\documentclass[12pt]{amsart}
\usepackage{amssymb,amsmath,amsfonts,latexsym}
\usepackage{bm}
\usepackage{mathtools}

\usepackage{array,graphics,color}
\numberwithin{equation}{section}

\newtheorem{dfn}{Definition}[section]
\newtheorem{thm}[dfn]{Theorem}
\newtheorem{lma}[dfn]{Lemma}

\newtheorem{crlre}[dfn]{Corollary}

\newtheorem{rmrk}[dfn]{Remark}

\DeclarePairedDelimiterX{\norm}[1]{\lVert}{\rVert}{#1}

\begin{document}

\title{Trace formula for contractions and it's representation in $\mathbb{D}$}

\author[Chattopadhyay] {Arup Chattopadhyay}
\address{Department of Mathematics, Indian Institute of Technology Guwahati, Guwahati, 781039, India}
\email{arupchatt@iitg.ac.in, 2003arupchattopadhyay@gmail.com}

\author[Sinha]{Kalyan B. Sinha}
\address{Theoretical Sciences Unit, Jawaharlal Nahru Centre For Advanced Scientific Research, Bangalore, 560064, India}
\email{kbs@jncasr.ac.in, kbs\_jaya@yahoo.co.in}

\subjclass[2010]{47A20, 47A55, 47A56, 47B10, 42B30, 30H10}

\keywords{Krein's trace formula, Spectral shift function, Self adjoint operators, Unitary operators, Contractions, Unitary dilations}

\begin{abstract}
	The aim of this article is twofold: give a short proof of the existence of real spectral shift function and the associated trace formula for a pair of contractions, the difference of which is trace-class and one of the two a strict contraction, so that the set of assumptions is minimal in comparison to those in all the existing proofs. The second one is to find a trace formula for differences of functions of contraction and its adjoint, in which case, the integral in the formula is over the unit disc and has an expression surprisingly similar to the Helton-Howe formula.   
\end{abstract}
\maketitle

\section{Introduction}

The notion of spectral shift function (SSF) for a trace-class perturbation of a self-adjoint operator originated in the work of Lifshitz \cite{L}, followed  by Krein in \cite{K}, in which it was shown that given a pair of (not necessarily bounded) self-adjoint operators $H$, $H_0$ such that $H-H_0$ is trace-class, there exists a unique real-valued function $L^1(\mathbb{R})$-function (SSF) $\xi$ satisfying \textquotedblleft \textit{Krein's Trace Formula}\textquotedblright:
\begin{equation} \label{eqsec1}
	\text{Tr}~\big\{\phi(H)-\phi(H_0)\big\} = \int_{\mathbb{R}} \phi'(\lambda)~\xi(\lambda)~d\lambda,
\end{equation}
for a large class of functions $\phi$. Krein's original proof uses complex function theory whereas an alternative proof of \eqref{eqsec1} by  Voiculescu \cite{V} uses the idea of quasi-diagonalization for bounded self-adjoint operators. Later, Sinha and Mohapatra (in \cite{SM1,SM2}) adapted the
quasi-diagonalization method to the cases of unbounded self-adjoint and unitary operators. For a pair $(U,U_0)$ of unitary operators with $U-U_0$ is trace-class, Krein's trace formula is
\begin{equation}\label{equ}
	\text{Tr}~\big\{\phi(U)-\phi(U_0)\big\} =\int_{0}^{2\pi} \phi'(\theta)~\tilde{\xi}(\theta)~d\theta
	= \int_{\mathbb{T}} \phi'(t)~\xi(t)~dt,
\end{equation}
for $\phi:\mathbb{T}\longrightarrow \mathbb{C}$ with an absolutely convergent Fourier series for $\phi'$ and with the spectral shift function $\xi$ a real-valued $L^1(\mathbb{T})$-function, unique upto an additive constant. The classes of function $\phi$ given here, for which formula \eqref{eqsec1} and \eqref{equ} hold, is not optimal and improved results in this direction was obtained by Peller and co-authors in \cite{AP1}.

For a pair of contractions $T$, $T_0$ with $T-T_0$ is trace-class,  Neidhardt \cite{N} initiated the study of trace formula, to be followed by others in \cite{AN1,MNV3}. In all these attempts, there were additional hypotheses involving the associated defect operators, for example all the defect operators $D_T:=(I-T^*T)^{\frac{1}{2}}$, $D_{T_0}:=(I-T_0T_0^*)^{\frac{1}{2}}$, $D_{T^*}:=(I-TT^*)^{\frac{1}{2}}$, $D_{T_0^*}:=(I-T_0T_0^*)^{\frac{1}{2}}$ 
were assumed to be trace-class as well in \cite{N}. On the other hand, as an example the authors in \cite{MNV3} constructs a pair $(T,T_0)$ of contractions such that $T-T_0$ is trace-class, and $T_0$ is Fredholm (that is, dim $Ker (T_0)$ = dim $Ker (T_0^*)$ $<\infty$) of index 0, for which a real spectral shift function is obtained. 

 The main purpose here is to prove the trace formula like that in \eqref{equ} with a real spectral shift function for a pair of contractions $T$, $T_0$, one of which is a strict contraction (that is, either $\|T_0\|$ or $\|T\|< 1$) and $T-T_0$ is trace-class \textit{with no further assumptions as in earlier studies}. As it turns out, the trace formula involves an integral, as on the right-hand side of \eqref{equ}, which is supported on $\mathbb{T}$. On the other hand, the spectra of both $T$ and $T_0$ are compact subsets of the closed unit disc $\overline{\mathbb{D}}$. Clearly, it will be desirable to have a trace formula, where the integral will be supported on the unit disc $\mathbb{D}$ instead of $\mathbb{T}$. For this, we need to consider functions of not just $T$ and $T_0$, but functions of $(T, T^*)$ and of $(T_0, T_0^*)$, and we show that the trace of the difference of such functions has an integral expression, in which the integral is supported on $\mathbb{D}$ (with planar Lebesgue measure), containing the natural harmonic extension of the real spectral shift function $\xi\in L^1(\mathbb{T})$ to $\tilde{\xi}$ on $\mathbb{D}$ by the Poisson integral formula. Furthermore, it turns out that this integral formula bears an intriguing resemblance to the Helton-Howe formula (see \cite{HH, CS}).

Here we denote by $\mathcal{B}(\mathcal{H})$ the collection of all bounded linear operators on $\mathcal{H}$ and $\mathcal{N}(B)$ denotes the null space of the operator $B \in \mathcal{B}(\mathcal{H})$, $\mathcal{B}_1(\mathcal{H})$ stands for the Banach space of trace-class operators on $\mathcal{H}$ and $\mathbb{R}$ and $\mathbb{N}$ for the set of real numbers and natural numbers respectively. The rest of the paper is organized as follows:
In section 2, we give the existence of real spectral shift function for a large class of pairs of contractions $(T, T_0)$ with trace class difference, and section 3 deals with extending the trace formula for functions of $(T, T^*)$ and $(T_0, T_0^*)$ with the integral over the unit disc $\mathbb{D}$.

\section{Trace formula for contractions}
In this section, a short and direct proof of the trace formula is given for a pair $(T,T_0)$ of contractions, with a real-valued SSF, when $T-T_0\in \mathcal{B}_1(\mathcal{H})$ and one of the two, say $T_0$, is a strict contraction, that is $0\leq T_0 <I$ or equivalently $D_{T_0}\geq \delta I>0$ for some $\delta>0$. In fact, the proof involved is symmetric with respect to the interchange of $T$ and $T_0$. The method employed here is similar, to that of \cite{MNV3} in the sense that the Nagy-dilation of a contraction in $\mathcal{H}$ to a unitary in $l^2_{\mathbb{Z}}(\mathcal{H})$ is used.  However, unlike in Theorem 9.4 of \cite{MNV3}, \emph{we do not assume that the differences of the dilated unitaries is trace-class, rather we prove it}. 

\begin{lma}\label{thm1}
	Let $A$ and $B$ be two positive contractions on a separable infinite-dimensional Hilbert space $\mathcal{H}$.
	\vspace{0.1in}
	
	\noindent $(i)$ If furthermore, $\mathcal{N} (B)= \{0\}$, then for every $f\in \mathcal{H}$
	\begin{equation}\label{eq0}
		\big(A-B\big) f =  \int_0^{+\infty} \textup{e}^{-tA}~\big(A^2-B^2\big) ~\textup{e}^{-tB} f ~dt,
	\end{equation}
	where the right hand side of \eqref{eq0} exists as an improper strong Riemann-Bochner integral.
	\vspace{0.1in}
	
	\noindent $(ii)$ If,  $B\geq \delta I> 0$ for some positive $\delta$ and if $\big(A^2-B^2\big)\in \mathcal{B}_1(\mathcal{H})$, then $\big(A-B\big)\in  \mathcal{B}_1(\mathcal{H})$.  
\end{lma}
\begin{proof}
	$(i)$~ By the definition of strong Riemann-Bochner integral \cite{Y} and by the observation that 
	\[
	\text{strong-}~\frac{d}{dt} \Big\{\textup{e}^{-tA}~\big(A-B\big) ~\textup{e}^{-tB} f \Big\} = -~\textup{e}^{-tA}~\big(A^2-B^2\big) ~\textup{e}^{-tB} f,  
	\]
	one gets that for every $f\in \mathcal{H}$,
	\begin{equation}\label{eq1}
		\big(A-B\big) f- \textup{e}^{-sA}~\big(A-B\big) ~\textup{e}^{-sB} f
		=  \int_0^{s} \textup{e}^{-tA}~\big(A^2-B^2\big) ~\textup{e}^{-tB} f ~dt
	\end{equation} 
	and for $s> s'> 0 $,
	\begin{equation}\label{eq2}
		-\int_{s'}^{s} \textup{e}^{-tA}~\big(A^2-B^2\big) ~\textup{e}^{-tB} f ~dt
		= \textup{e}^{-sA}~\big(A-B\big) ~\textup{e}^{-sB} f -\textup{e}^{-s'A}~\big(A-B\big) ~\textup{e}^{-s'B} f.
	\end{equation}
	By hypothesis $\mathcal{N} (B)= \{0\}$, this implies by the spectral theorem that
	\[
	\left\|\textup{e}^{-sB} f\right\|^2 = \int_{0-}^{1} \textup{e}^{-2s\lambda}~\left\|\mathcal{E}_{B}(d\lambda)f\right\|^2 \longrightarrow 0 \quad \text{as}\quad s\longrightarrow +\infty,
	\]
	by the dominated convergence theorem, where $\mathcal{E}_B(\cdot)$ is the spectral measure corresponding to the operator $B$, and hence, the right hand side of \eqref{eq2} converges to $0$ as $s'\longrightarrow +\infty$ for every $f\in \mathcal{H}$. This proves the existence of the improper strong Riemann-Bochner integral which appears in \eqref{eq0} as well as the equality in
	\eqref{eq0}. 
	\vspace{0.1in}
	
	$(ii)$~ By $(i)$, we have an identity in  $\mathcal{B}(\mathcal{H})$: 
	\begin{equation}\label{sec2eq}
		A-B =  \int_0^{+\infty} \textup{e}^{-tA}~\big(A^2-B^2\big) ~\textup{e}^{-tB} ~dt,
	\end{equation}  
	since the hypothesis in $(ii)$ in particular implies that $\mathcal{N} (B)= \{0\}$. Moreover,  the right hand side integral in \eqref{sec2eq} converges in $\mathcal{B}_1(\mathcal{H})$-topology since $\left\| \textup{e}^{-tA}\right\|\leq 1$ and $\left\| \textup{e}^{-tB}\right\|\leq  \textup{e}^{-\delta t}$.    
\end{proof}

The following corollary is a simple consequence of the above lemma.
\begin{crlre}\label{cor1}
	Let $T$ and $T_0$ be two contractions on a separable infinite dimensional Hilbert space $\mathcal{H}$ such that one of $T$ and $T_0$ is a strict contraction and $T-T_0\in \mathcal{B}_1(\mathcal{H})$. Then both $D_T-D_{T_0}$ and $D_{T^*}-D_{T_0^*}$ are trace class operators. 
\end{crlre}
\begin{proof}
	Without loss of generality we assume that $\|T_0\|< 1$. Then it follows that there exists a positive $\delta ~(>0)$ such that $D_{T_0}\geq \delta I > 0$ and $D_{T_0^*}\geq \delta I > 0$. Now, by hypothesis we conclude that 
	\begin{equation*}
		D_T^2-D_{T_0}^2= (T_0-T)^*T_0 + T^*(T_0-T)\quad \text{and}\quad D_{T^*}^2-D_{T_0^*}^2= (T_0-T)T_0^* + T(T_0-T)^*
	\end{equation*}
	are trace class operators and therefore the proof follows by applying Theorem~\ref{thm1} $(ii)$ with $A=D_T$ and $B=D_{T_0}$ ; and  $A=D_{T^*}$ and $B=D_{T_0^*}$ respectively. 
\end{proof}
Now we are in a position to state and prove our main result in this section. For this, we set  $$\mathcal{G}:=\{f(z)=\sum\limits_{k=0}^{\infty}a_kz^k:~~ z\in \mathbb{D}\quad a_k\in \mathbb{C}\quad  \text{and} \quad \sum\limits_{k=0}^{\infty}|ka_k|<\infty\}.$$
\begin{thm}\label{mainthm1}
	Let $T$ and $T_0$ be two contractions on a separable infinite dimensional Hilbert space $\mathcal{H}$ such that $T_0$ is a strict contraction and $T-T_0\in \mathcal{B}_1(\mathcal{H})$. Then there exists a unique (up to an additive constant) real valued function $\xi\in L^1([0,2\pi])$ such that 
	\begin{equation}\label{equ7}
		\textup{Tr}_{\mathcal{H}}~\big\{\phi(T)-\phi(T_0)\big\} = \int_{0}^{2\pi} \frac{d}{dt}\big\{\phi(\textup{e}^{it})\big\}~\xi(t)~dt
	\end{equation}
	for every $\phi$ in $\mathcal{G}$. 
\end{thm}
\begin{proof}
	Using the construction of the Sch$\ddot{a}$ffer matrix unitary dilation of contractions (see \cite{NF}, chapter I, section 5), we dilate $T$ and $T_0$ to the corresponding unitary operators $U_T$ and $U_{T_0}$ respectively on the same Hilbert space $\ell_{\mathbb{Z}}^2(\mathcal{H})$, that is
	\begin{equation*}
		\begin{split}
			T^n=P_{\mathcal{H}} U_T^n|_{\mathcal{H}}\quad \text{and}\quad T_0^n=P_{\mathcal{H}} U_{T_0}^n|_{\mathcal{H}}\quad \text{for}\quad n\geq 1,
		\end{split}
	\end{equation*}
	where $P_{\mathcal{H}}$ is the orthogonal projection of $\ell_{\mathbb{Z}}^2(\mathcal{H})$ onto 
	$\mathcal{H}$ and furthermore the explicit expressions of $U_T$ and $U_{T_0}$ are as follows:
	\begin{align}\label{eqq9}
		\nonumber & U_T\big(\{x_n\}_{n\in \mathbb{Z}}\big) := \\
		& \Big\{\cdots, x_{-2},x_{-1}, D_Tx_0-T^*x_1, \boxed{Tx_0+D_{T^*}x_1},x_2,x_3,\ldots\Big\},\quad \{x_n\}_{n\in \mathbb{Z}} \in \ell_{\mathbb{Z}}^2(\mathcal{H}),
	\end{align}
	and 
	\begin{align}\label{eqq10}
		\nonumber & U_{T_0}\big(\{x_n\}_{n\in \mathbb{Z}}\big) := \\
		&\Big\{\cdots, x_{-2},x_{-1}, D_{T_0}x_0-T_0^*x_1, \boxed{T_0x_0+D_{T_0^*}x_1},x_2,x_3,\ldots\Big\}, \quad \{x_n\}_{n\in \mathbb{Z}} \in \ell_{\mathbb{Z}}^2(\mathcal{H}),
	\end{align} 
	where the boxed entry in the above sequences corresponds to the term indexed by 0 and we identify the Hilbert space $\mathcal{H}$ as a closed subspace of $\ell_{\mathbb{Z}}^2(\mathcal{H})$ consisting of sequences $\{x_n\}_{n\in \mathbb{Z}}$ such that $x_n=0$ for $n\neq 0$. Thus for $\{x_n\}_{n\in \mathbb{Z}} \in \ell_{\mathbb{Z}}^2(\mathcal{H})$ we have 
	\begin{align}\label{eq8}
		\nonumber &\big(U_T-U_{T_0}\big)\big(\{x_n\}_{n\in \mathbb{Z}}\big) :=\\
		& \Big\{\cdots,0,0, ~\left(D_T-D_{T_0}\right)x_0-\left(T^*-T_0^*\right)x_1,~ \boxed{(T-T_0)x_0+\left(D_{T^*}-D_{T_0^*}\right)x_1},~0,0,\cdots\Big\}.
	\end{align}
	In other words, there will be exactly four non-zero entries namely $T-T_0$ at the $(0,0)$ position, $D_{T^*}-D_{T_0^*}$ at the $(0,1)$ position, $D_T-D_{T_0}$ at the $(-1,0)$ position and $-(T^*-T_0^*)$ at the $(-1,1)$ position survive in the block matrix representation of $U_T-U_{T_0}$. On the other hand by hypothesis since $T_0$ is a strict contraction and  $T-T_0$ is trace class, then from the Corollary~\ref{cor1} we conclude that both $D_T-D_{T_0}$ and $D_{T^*}-D_{T_0^*}$ are trace class operators. Therefore it follows from \eqref{eq8} that all non-zero entries in the block matrix representation of $U_T-U_{T_0}$ are trace class and hence $U_T-U_{T_0}$ \emph{is also a trace class operator} in $\ell_{\mathbb{Z}}^2(\mathcal{H})$. This implies that $U_{T}^n-U_{T_0}^n\in \mathcal{B}_1(\ell_{\mathbb{Z}}^2(\mathcal{H}))$ for all $n\geq 1$. Moreover, it is also clear from equations \eqref{eqq9} and \eqref{eqq10} that both $U_T$ and $U_{T_0}$ are upper triangular matrices with the only non-zero diagonal entries $T$ and $T_0$ and hence both $U_T^n$ and $U_{T_0}^n$ are also upper triangular matrices with the only non-zero diagonal entries $T^n$ and $T_0^n$ respectively for any $n\geq 1$. Thus
	\begin{equation*}
		T^n-T_0^n = P_{\mathcal{H}}\{U_T^n-U_{T_0}^n\}|_{\mathcal{H}}, \quad \text{and}\quad \textup{Tr}_{\mathcal{H}}~\big\{T^n-T_0^n\big\} =   \textup{Tr}_{\ell_{\mathbb{Z}}^2(\mathcal{H})}~\big\{U_{T}^n-U_{T_0}^n\big\},\quad n\geq 0,
	\end{equation*}
	and hence
	\begin{equation}\label{eq26}
		\begin{split}
			& p(T)-p(T_0) = P_{\mathcal{H}}\{p(U_T)-p(U_{T_0})\}|_{\mathcal{H}}, \quad \text{and}\\
			&  \textup{Tr}_{\mathcal{H}}~\big\{p(T)-p(T_0)\big\} =   \textup{Tr}_{\ell_{\mathbb{Z}}^2(\mathcal{H})}~\big\{p(U_{T})-p(U_{T_0})\big\}
		\end{split}
	\end{equation}
	for any polynomial $p(\cdot)$ in $\mathbb{D}$. Next for $\phi(z)=\sum\limits_{k=0}^{\infty}a_kz^k\in\mathcal{G}$, and if we denote $p_n(z)=\sum\limits_{k=0}^{n}a_kz^k$, then it is easy to check that $p_n(T)$, $p_n(T_0)$ and $p_n(U_T)$, $p_n(U_{T_0})$ converge to $\phi(T)$, $\phi(T_0)$ and $\phi(U_T)$, $\phi(U_{T_0})$ respectively in operator norm. Therefore from the above equation \eqref{eq26} we conclude that 
	\begin{equation}\label{eq27}
		\phi(T)-\phi(T_0) = P_{\mathcal{H}}\{\phi(U_T)-\phi(U_{T_0})\}|_{\mathcal{H}}.
	\end{equation}
	Furthermore, since $\sum\limits_{k=0}^{\infty}|ka_k|<\infty$ as $\phi\in \mathcal{G}$ and $U_T-U_{T_0}\in \mathcal{B}_1(\ell_{\mathbb{Z}}^2(\mathcal{H})),$ it follows that
	\begin{equation}\label{eq34}
		\begin{split}
			&\Big\|\big\{\phi(U_T)-\phi(U_{T_0})\big\} - \big\{p_n(U_T)-p_n(U_{T_0})\big\}\Big\|_1\\
			& \leq \left\|\sum_{k=n+1}^{\infty}a_k\big\{U_T^{k}-U_{T_0}^k\big\}\right\|_1 = \left\|\sum_{k=n+1}^{\infty}a_k \Big\{\sum_{j=0}^{k-1}U_T^{k-j-1}(U_T-U_{T_0})U_{T_0}^j\Big\}\right\|_1\\
			& \leq \|U_T-U_{T_0}\|_1 \sum_{k=n+1}^{\infty}|ka_k|
			\longrightarrow 0 \quad \text{as} \quad n\longrightarrow \infty.
		\end{split}
	\end{equation}
	Thus $\big\{\phi(U_T)-\phi(U_{T_0})\big\}\in \mathcal{B}_1(\ell_{\mathbb{Z}}^2(\mathcal{H}))$ and from \eqref{eq27} we conclude that $\big\{\phi(T)-\phi(T_0)\big\}$ is a trace class operator and
	\begin{equation}\label{eq9}
		\text{Tr}_{\mathcal{H}}~\big\{\phi(T)-\phi(T_0)\big\}
		= \text{Tr}_{\ell_{\mathbb{Z}}^2(\mathcal{H})}~\Big\{P_{\mathcal{H}}\{\phi(U_T)-\phi(U_{T_0})\}P_{\mathcal{H}}\Big\}.
	\end{equation}
Furthermore, since $\big\{\phi(U_T)-\phi(U_{T_0})\big\}\in \mathcal{B}_1(\ell_{\mathbb{Z}}^2(\mathcal{H}))$, and since all the diagonal entries of  $P_{\mathcal{H}}^{\perp}\big\{\phi(U_T)-\phi(U_{T_0})\big\}P_{\mathcal{H}}^{\perp}$ in the  Sch$\ddot{a}$ffer-dilation basis are zero, it follows that 
$$\text{Tr}_{\ell_{\mathbb{Z}}^2(\mathcal{H})}~\Big\{P_{\mathcal{H}}^{\perp}\{\phi(U_T)-\phi(U_{T_0})\}P_{\mathcal{H}}^{\perp}\Big\}=0$$ and hence we get
	\begin{equation}\label{eq13}
		\text{Tr}_{\mathcal{H}}~\big\{\phi(T)-\phi(T_0)\big\}
		= \textup{Tr}_{\ell_{\mathbb{Z}}^2(\mathcal{H})}~\big\{\phi(U_{T})-\phi(U_{T_0})\big\}.
	\end{equation} 
	Note that $U_T$ and $U_{T_0}$ are two unitary operators on
	$l_{\mathbb{Z}}^2(\mathcal{H})$ such that $U_T-U_{T_0}\in \mathcal{B}_1(\ell_{\mathbb{Z}}^2(\mathcal{H}))$ and therefore by Krein's trace formula corresponding to the pair $(U_T,U_{T_0})$ \cite{AP1, K1, MNV3, SM1} there exists 
	a real-valued $L^1([0,2\pi])$-function $\xi$ (known as the SSF, corresponding to the pair $(U_T,U_{T_0})$) such that
	\begin{equation} \label{eq28}
		\text{Tr}_{\ell_{\mathbb{Z}}^2(\mathcal{H})}~\big\{\phi(U_T)-\phi(U_{T_0})\big\} = \int_{0}^{2\pi} \frac{d}{dt}\big\{\phi(\textup{e}^{it})\big\}~\xi (t)~dt
	\end{equation}
	for every $\phi$ in $\mathcal{G}$.
	Finally, combining \eqref{eq13} and \eqref{eq28} we arrive at: 
	\begin{equation}\label{eqr29}
		\begin{split}
			\text{Tr}_{\mathcal{H}}~\big\{\phi(T)-\phi(T_0)\big\}
			= \int_{0}^{2\pi} \frac{d}{dt}\big\{\phi(\textup{e}^{it})\big\}~\xi(t)~dt
		\end{split}
	\end{equation}
	for every $\phi$ in $\mathcal{G}$. Since $\xi\in L^1(\mathbb{T})$, real-valued, it follows that if there is another such $\eta$ such that \eqref{eqr29} is satisfied with $\eta$ replacing $\xi$, then this will imply that
	\begin{equation*}
		\int_0^{2\pi} \frac{d}{dt}\big\{\phi(\textup{e}^{it})\big\}~\big[\xi(t)-\eta(t)\big]dt =0\quad \text{for every} \quad \phi \in \mathcal{G},
	\end{equation*}
in particular
\begin{equation*}
 	\int_0^{2\pi} in~\textup{e}^{int}~\big[\xi(t)-\eta(t)\big]dt =0\quad \forall \quad n\geq 0.
\end{equation*}
Taking the complex-conjugate of this relation, we arrive at 
\begin{equation*}
		\int_0^{2\pi} \textup{e}^{int}~\big[\xi(t)-\eta(t)\big]dt =0\quad \forall \quad n\in \mathbb{Z}\setminus \{0\}.
\end{equation*}
Since $\xi-\eta\in L^1(\mathbb{T})$, this would imply that $\xi-\eta=constant$ or that $\xi$ is unique modulo an additive constant. This completes the proof. 
\end{proof}

\begin{rmrk}
	$(i)$~The SSF $\xi$ for a pair $(U,U_0)$ of unitaries is real-valued $L^1(\mathbb{T})$-function, by its construction (\cite{K,MNV3,SM1}). However, while one can write 
	\begin{equation*}
		\textup{Tr}~\big\{U^n-U_0^n\big\}= 	\int_0^{2\pi} in~\textup{e}^{int} \xi(t) dt\quad \text{for all}\quad n\in \mathbb{Z}
	\end{equation*}
one can not, in general, do the same for $\big\{T^n-T_0^n\big\}$ for $n$ a negative integer because $T$ (or $T_0$) need not be invertible. On the other hand, the same purpose will be served if one takes the adjoint instead and note that: if $\|T\|\leq 1$, ~$\|T_0\|<1$ and $T-T_0\in \mathcal{B}_1(\mathcal{H})$, the same is true for the pair $(T^*,T_0^*)$ replacing the pair $(T,T_0)$. This leads to (recalling that $\xi$ is real-valued)
\begin{equation}\label{eq32}
	\begin{split}
		& \textup{Tr}~\big\{T^{*^n}-T_0^{*^n}\big\}  =\overline{\textup{Tr}~\big\{T^{^{n}}-{T_0}^{n}\big\}}=\overline{ \int_{0}^{2\pi} in~\textup{e}^{int}~\xi (t)~dt}=-in \int_{0}^{2\pi} \textup{e}^{-int}~\xi (t)~dt,
	\end{split}
\end{equation}
for $n\in \mathbb{N}$, though the functions $\big\{\textup{e}^{-int}|~n\in \mathbb{N}\big\}$ do not belong to $\mathcal{G}$. We shall exploit this simple observation in the next section to construct a different kind of trace formula for the difference of functions of $(T,T^*)$ and of $(T_0,T_0^*)$ to obtain an integral expression with support on $\mathbb{D}$ instead of on $\mathbb{T}$.
\vspace{0.1in}

$(ii)$~As we have observed earlier that, for $n\in \mathbb{N}$
\begin{equation*}
		T^{*^n}-T_0^{*^n} = P_{\mathcal{H}}\{U_T^{*^n}-U_{T_0}^{*^n}\}|_{\mathcal{H}} = P_{\mathcal{H}}\{U_T^{-n}-U_{T_0}^{-n}\}|_{\mathcal{H}}.
\end{equation*}
 But instead one could have dilated the pair $(T^*,T_0^*)$ of contractions to obtain for $n\in \mathbb{N}$, 
 	$T^{*^n}=P_{\mathcal{H}} U_{T^*}^{n}|_{\mathcal{H}}$, $T_0^{*^n}=P_{\mathcal{H}} U_{T_0^*}^{n}|_{\mathcal{H}}$.
Thus though 
\begin{equation*}
	\textup{Tr}_{\mathcal{H}}~\big\{T^{*^n}-T_0^{*^n}\big\} = \textup{Tr}_{\ell^2_{\mathbb{Z}}(\mathcal{H})}~\big\{U_T^*{^{^n}}-U_{T_0}^{*^{n}}\big\}  
	 =-in \int_{0}^{2\pi} \textup{e}^{-int}~\xi (t)~dt \quad \text{for} \quad n\in \mathbb{N},
\end{equation*}
it is also equal to 
\begin{equation*}
	\textup{Tr}_{\ell^2_{\mathbb{Z}}(\mathcal{H})}~\big\{U_{T^*}^{n}-U_{T_0^*}^{n}\big\}  
	=in \int_{0}^{2\pi} \textup{e}^{int}~\chi (t)~dt,
\end{equation*}
with both $\xi$ and $\chi$ are real-valued $L^1(\mathbb{T})$-functions. It may be noted that while $\big\{U_{T^*}-U_{T_0^*}\big\}$ is an upper triangular matrix-operator, $\big(U_T^*-U_{T_0}^*\big)=\big(U_T-U_{T_0}\big)^*$ is a lower-triangular one. Thus, we do not expect the two SSF's $\xi$ and $\chi$ to be related. However, it is easy to see that:
\begin{equation*}
	-in \int_{0}^{2\pi} \textup{e}^{-int}~\xi (t)~dt = in \int_{0}^{2\pi} \textup{e}^{int}~\chi (t)~dt
\end{equation*}
or $\hat{\chi}(n)=-\hat{\xi}(-n)$ for each $n\in \mathbb{N}\setminus \{0\}$ or equivalently $\chi(t) +\xi(-t)= constant$.  
\end{rmrk}

\section{Trace formula with support on the disc $\mathbb{D}$}
The SSF $\xi$ in \eqref{eqsec1} corresponding to a pair of self-adjoint operators $(H,H_0)$ such that $H-H_0\in \mathcal{B}_1(\mathcal{H})$ is supported on a subset of $\mathbb{R}$ which contains the spectrum of $H$ as well as the spectrum of $H_0$. The same is true for SSF $\eta$ in \eqref{equ} corresponding to a pair of unitary operators $(U,U_0)$ with trace class difference, that is, the support of $\eta$ lies in $\mathbb{T}$ and contain the spectrum of $U$ and $U_0$. Therefore it is expected while dealing with a pair of contractions $(T,T_0)$ with $T-T_0\in \mathcal{B}_1(\mathcal{H})$, the shift function corresponding to that pair $(T,T_0)$ should also be supported on a subset of the closed unit disc $\overline{\mathbb{D}}$ which contains the spectrum of $T$ and $T_0$. But in \eqref{equ7} we see that this is not the case, that is, the shift function  $\xi$ corresponding to the pair $(T,T_0)$ is supported on the unit circle $\mathbb{T}$ whereas the spectrum of $T$ and $T_0$ is contained in $\overline{\mathbb{D}}$. Our next part of the analysis is devoted to obtaining an appropriate justification of the fact mentioned above.
Next observe that the Lebesgue measure in $\mathbb{C}$, restricted to $\mathbb{T}$ is zero and therefore, to obtain a shift function supported on $\overline{\mathbb{D}}$ and satisfying a formula like  \eqref{equ7} corresponding to a pair of contractions $(T,T_0)$, it is necessary to consider functions of the operator along with its adjoint instead of that of the operator alone. To obtain an extension of $L^1(\mathbb{T})$-function as a harmonic function into the interior $\mathbb{D}$ of $\mathbb{T}$, we next use the the Poisson integral representation. Let $f\in L^1(\mathbb{T})$, then the Poisson integral of $f$ is denoted by $Pf$ and is defined by
\begin{equation}\label{eq14}
	(Pf)(z) =  \frac{1}{2\pi}\int_0^{2\pi}\frac{1-|z|^2}{\left|\textup{e}^{it}-z\right|^2}~f(\textup{e}^{it})~dt,
\end{equation}
where $z\in \mathbb{D}$ and $\frac{dt}{2\pi}$ is the normalized Lebesgue measure on $\mathbb{T}$. Now one can verify that for $z\in \mathbb{D}$,
\begin{equation}\label{eq33}
	\frac{1-|z|^2}{\left|\textup{e}^{it}-z\right|^2}
	= 1 + \sum_{n=1}^{\infty} \bar{z}^n\textup{e}^{int}+\sum_{n=1}^{\infty} z^n\textup{e}^{-int},
\end{equation} 
and therefore combining equations \eqref{eq14} and \eqref{eq33} we conclude that $Pf$ is a harmonic function on $\mathbb{D}$ and that
\begin{equation}\label{eq16}
	(Pf)(z) =  \hat{f}(0) + \sum_{n=1}^{\infty} \hat{f}(-n)\bar{z}^n+\sum_{n=1}^{\infty} \hat{f}(n) z^n,\quad \text{for}\quad z\in \mathbb{D},
\end{equation}
where $\hat{f}(n)$ is the $n$-the Fourier coefficient of $f$ given by 
\[
\hat{f}(n) = \frac{1}{2\pi}\int_0^{2\pi} f(\textup{e}^{it})~\textup{e}^{-int}~dt, \quad \text{where} \quad n\in \mathbb{Z}.
\]
In other words, through Poisson integral transform, we extend an $L^1(\mathbb{T})$- function defined on the boundary to the interior of $\mathbb{T}$, that is on $\mathbb{D}$ as a harmonic function.    One of the oldest results about the boundary behavior of $Pf$ is due to Fatou, and it says that the original function $f$ is retrieved almost everywhere as a boundary value of $Pf$.
\begin{thm}
	Let  $f\in L^1(\mathbb{T})$, then 
	\[
	\lim_{r\longrightarrow 1-}(Pf)(r\textup{e}^{it})= f(\textup{e}^{it})
	\]
	for all $\textup{e}^{it}\in \mathbb{T}$ except possibly on a set of  measure zero.
\end{thm}
For more on the Poisson integral and related matter see \cite{DU, Rudin} and the references therein. Note that we prove our main theorem (see Theorem~\ref{mainthmthirdsec}) in this section for the following class of functions
\begin{equation*}
	\widetilde{\mathcal{G}}:=\Big\{\psi:\mathbb{T}\longrightarrow \mathbb{C}~\Big|~\sum_{n=-\infty}^{\infty}|n\hat{\psi}(n)|<\infty, \quad \text{where} \quad \hat{\psi}(n) = \frac{1}{2\pi}\int_0^{2\pi} \psi(\textup{e}^{it})~\textup{e}^{-int}~dt, \quad n\in \mathbb{Z}\Big\}.
\end{equation*}
Let $\psi \in \widetilde{\mathcal{G}}$ and set for $z\in \overline{\mathbb{D}}=\mathbb{D}\cup \mathbb{T}$,
\begin{equation}\label{eq15}
	\widetilde{\psi}(z,\bar{z})
	= \hat{\psi}(0) + \sum\limits_{n=1}^{\infty} \hat{\psi}(-n)\bar{z}^n+\sum\limits_{n=1}^{\infty} \hat{\psi}(n) z^n.
\end{equation}
Now it is important to observe that the Poisson integral transform of $\psi$, that is $P\psi$ may not exists for $z\in\mathbb{T}$ but
the right hand side of \eqref{eq15} makes sense for $z\in \mathbb{T}$ because of the extra assumption on $\psi$, namely $\sum\limits_{n=-\infty}^{\infty}|n\hat{\psi}(n)|<\infty.$ This implies that  $\sum\limits_{n=-\infty}^{\infty}|\hat{\psi}(n)|<\infty$ and furthermore the equation \eqref{eq16} implies that the extension \eqref{eq15} of $\psi$ to the interior $\mathbb{D}$ of $\mathbb{T}$ is same as  the Poisson integral transform of $\psi$, that is $\widetilde{\psi}(z,\bar{z})=(P\psi)(z)$ for $z\in \mathbb{D}$. Next we need the following useful lemma towards obtaining our main result in this section. 
\begin{lma}\label{finallemma}
	Let $f:\mathbb{T}\longrightarrow \mathbb{C}$ be such that $\sum\limits_{n=-\infty}^{\infty}|\hat{f}(n)|<\infty$, and let $g\in L^1([0,2\pi])$. Then 
	\begin{equation*}
		\frac{1}{2\pi}\int_{0}^{2\pi} f(\textup{e}^{it})~g(t)~dt = \sum_{n=-\infty}^{\infty}\hat{f}(n)~\hat{g}(-n).
	\end{equation*}
\end{lma}
\begin{proof}
	Since $\sum\limits_{n=-\infty}^{\infty}|\hat{f}(n)|<\infty$, then it is easy to verify that $f$ is a bounded continuous function on $\mathbb{T}$ and furthermore we have that
	\begin{equation}\label{eq17}
		f(\textup{e}^{it}) = \sum_{n=-\infty}^{\infty} \hat{f}(n)~\textup{e}^{int},
	\end{equation}
	where the series in the right hand side converges uniformly on $\mathbb{T}$. Thus using the above expression  \eqref{eq17} of $f$ we conclude that
	\begin{equation*}
		\begin{split}
			\frac{1}{2\pi}\int_{0}^{2\pi} f(\textup{e}^{it})~g(t)~dt & = \frac{1}{2\pi}\int_{0}^{2\pi}\left(\sum_{n=-\infty}^{\infty} \hat{f}(n)~\textup{e}^{int}\right)~~g(t)~dt\\
			& = \sum_{n=-\infty}^{\infty} \hat{f}(n)~\left(\frac{1}{2\pi}\int_{0}^{2\pi} g(t)~\textup{e}^{int}~dt\right) = \sum_{n=-\infty}^{\infty}\hat{f}(n)~\hat{g}(-n),
		\end{split}
	\end{equation*}
	where we have applied Fubini's theorem to get the second equality since $\sum\limits_{n=-\infty}^{\infty}|\hat{f}(n)|<\infty$ and $g\in L^1([0,2\pi])$. This completes the proof. 
\end{proof}
Suppose $T\in \mathcal{B}(\mathcal{H})$ be such that $\|T\|\leq 1$ and using the representation \eqref{eq15} of $\psi\in \widetilde{\mathcal{G}}$ we set 
\begin{equation*}
	\widetilde{\psi}(T,T^*)= \hat{\psi}(0)I + \sum\limits_{n=1}^{\infty} \hat{\psi}(-n)~T^{*^n}+\sum\limits_{n=1}^{\infty} \hat{\psi}(n)~T^{^n},
\end{equation*}
where the right hand side converges in operator norm since $\sum\limits_{n=-\infty}^{\infty}|\hat{\psi}(n)|<\infty$. Now we are in a position to state and prove our main result in this section for a class of function $\widetilde{\mathcal{G}}$. 
\begin{thm} \label{mainthmthirdsec}
	Let $T$ and $T_0$ be two contractions on a separable infinite dimensional Hilbert space $\mathcal{H}$ such that $T_0$ is a strict contraction and $T-T_0\in \mathcal{B}_1(\mathcal{H})$. Then for $\psi\in\widetilde{\mathcal{G}}$, the operator $\widetilde{\psi}(T,T^*)-\widetilde{\psi}(T_0,T_0^*)$ is trace class and
	\begin{equation}\label{eq25}
		\textup{Tr}_{\mathcal{H}}~\Big\{\widetilde{\psi}(T,T^*)-\widetilde{\psi}(T_0,T_0^*)\Big\}
		= \int_{\mathbb{D}}J\big(\widetilde{\xi},\widetilde{\psi}\big)(z,\bar{z})~dz\wedge d\bar{z},
	\end{equation}
	where $\widetilde{\psi}$ is as in \eqref{eq15}, $\xi$ is the shift function corresponding to the pair $(T,T_0)$ as in \eqref{equ7} and $\widetilde{\xi}(z,\bar{z})=(P\xi)(z)$ for $z\in \mathbb{D}$,  $J\big(\widetilde{\xi},\widetilde{\psi}\big) = \frac{\partial  \widetilde{\xi}}{\partial z}\frac{\partial  \widetilde{\psi}}{\partial \bar{z}} -\frac{\partial  \widetilde{\psi}}{\partial z}\frac{\partial  \widetilde{\xi}}{\partial \bar{z}}$ is the Jacobian of $\widetilde{\xi}$ and $\widetilde{\psi}$ on $\mathbb{D}$,
	$dz\wedge d\bar{z}$ is the Lebesgue measure on $\mathbb{D}$ and the integral in the right hand side of \eqref{eq25} is to be interpreted as an improper Riemann-Lebesgue integral
	$$\lim\limits_{R\uparrow 1} \int\limits_{\{z||z|\leq R<1\}}J\big(\widetilde{\xi},\widetilde{\psi}\big)(z,\bar{z})~dz\wedge d\bar{z}.$$	 
\end{thm}
\begin{proof}
	First we note that 
	\begin{equation}\label{eq18}
		\begin{split}
			& \widetilde{\psi}(T,T^*) - \widetilde{\psi}(T_0,T_0^*) = \sum\limits_{n=1}^{\infty} \hat{\psi}(-n)~\big\{T^{*^n}-T_0^{*^n}\big\} +\sum\limits_{n=1}^{\infty} \hat{\psi}(n)~ \big\{T^{^n}-T_0^n\big\}\\
			& = \sum\limits_{n=1}^{\infty} \sum_{j=0}^{n-1}
			\hat{\psi}(-n) ~T^{*^{n-j-1}}(T^*-T_0^*)T_0^{*^j} + \sum\limits_{n=1}^{\infty} \sum_{j=0}^{n-1}
			\hat{\psi}(n) ~T^{{n-j-1}}(T-T_0)T_0^{j}.
		\end{split}
	\end{equation}
	Since $T-T_0\in \mathcal{B}_1(\mathcal{H})$ and $\sum\limits_{n=-\infty}^{\infty}|n\hat{\psi}(n)|<\infty$ because $\psi\in \tilde{\mathcal{G}}$, then both the series in the right hand side of the above equation \eqref{eq18}  converge in trace norm and hence $\widetilde{\psi}(T,T^*) - \widetilde{\psi}(T_0,T_0^*)$ is trace class and furthermore we have 
	\begin{equation*}
		\big\|\widetilde{\psi}(T,T^*) - \widetilde{\psi}(T_0,T_0^*)\big\|_1 \leq \left(\sum\limits_{n=-\infty}^{\infty}|n\hat{\psi}(n)|\right) \|T-T_0\|_1<\infty,
	\end{equation*}
	and 
	\begin{equation}\label{eq19}
		\textup{Tr}_{\mathcal{H}}~\Big\{\widetilde{\psi}(T,T^*) - \widetilde{\psi}(T_0,T_0^*)\Big\} = \sum\limits_{n=1}^{\infty} \hat{\psi}(-n)~\textup{Tr}_{\mathcal{H}} ~\Big\{T^{*^n}-T_0^{*^n}\Big\} +\sum\limits_{n=1}^{\infty} \hat{\psi}(n)~ \textup{Tr}_{\mathcal{H}} ~\Big\{T^{^n}-T_0^n\Big\}.
	\end{equation}
	Next by combining \eqref{eq32} and \eqref{eq19} we get
	\begin{equation*}
		\begin{split}
			\textup{Tr}_{\mathcal{H}} ~\Big\{\widetilde{\psi}(T,T^*) - \widetilde{\psi}(T_0,T_0^*)\Big\} & = \sum\limits_{n=1}^{\infty} \hat{\psi}(-n)~-in \int_{0}^{2\pi} \textup{e}^{-int}~\xi(t)~dt +\sum\limits_{n=1}^{\infty} \hat{\psi}(n)~in \int_{0}^{2\pi} \textup{e}^{int}~\xi(t)~dt\\
			& = 2\pi i \sum\limits_{n=1}^{\infty} \hat{\psi}(-n)~(-n)~\hat{\xi}(n) +\sum\limits_{n=1}^{\infty} \hat{\psi}(n)~(n)~\hat{\xi}(-n)\\
			& = 2\pi i \sum_{n=-\infty}^{\infty} n~\hat{\psi}(n)~\hat{\xi}(-n) = 2\pi  \sum_{n=-\infty}^{\infty} ~\hat{\psi'}(n)~\hat{\xi}(-n)
		\end{split}
	\end{equation*}
	which by applying Lemma~\ref{finallemma} corresponding to  $f=\psi'$ and $g=\xi$ yields
	\begin{equation}\label{eq21}
		\textup{Tr}_{\mathcal{H}} ~\Big\{\widetilde{\psi}(T,T^*) - \widetilde{\psi}(T_0,T_0^*)\Big\} =\int_{0}^{2\pi} \frac{d}{dt}\big\{\psi(\textup{e}^{it})\big\}~\xi(t)~dt =  2\pi i \sum_{n=-\infty}^{\infty} ~n~\hat{\psi}(n)~\hat{\xi}(-n).
	\end{equation}
Also since for $z\in \mathbb{D}$, $|z|<1$, the Jacobian has the expression,
	\begin{equation*}
		\begin{split}
			J\big(\widetilde{\xi},\widetilde{\psi}\big) & = \frac{\partial  \widetilde{\xi}}{\partial z}\frac{\partial  \widetilde{\psi}}{\partial \bar{z}} -\frac{\partial  \widetilde{\psi}}{\partial z}\frac{\partial  \widetilde{\xi}}{\partial \bar{z}}\\
			& = \left(\sum_{n=1}^{\infty} n \hat{\xi}(n) z^{n-1}\right)\left(\sum_{m=1}^{\infty} m\hat{\psi}(-m)\bar{z}^{m-1}\right) - \left(\sum_{m=1}^{\infty} m \hat{\xi}(-m) \bar{z}^{m-1}\right)\left(\sum_{n=1}^{\infty} n\hat{\psi}(n)z^{n-1}\right)\\
			& =\sum_{n,m=1}^{\infty}nm\big\{\hat{\xi}(n)\hat{\psi}(-m)-\hat{\xi}(-m)\hat{\psi}(n)\big\}~z^{n-1}\bar{z}^{m-1},
		\end{split}
	\end{equation*}
where the double series converges absolutely and uniformly for $|z|\leq R<1$. Therefore for $R<1$,
	\begin{equation}\label{eq22}
		\begin{split}
			\int_{\mathbb{D}_{_R}} J\big(\widetilde{\xi},\widetilde{\psi}\big)~dz\wedge d\bar{z}
			& = \int_{\mathbb{D}_{_R}} \left(\sum_{n,m=1}^{\infty}nm\big\{\hat{\xi}(n)\hat{\psi}(-m)-\hat{\xi}(-m)\hat{\psi}(n)\big\}\right)~z^{n-1}\bar{z}^{m-1}~dz\wedge d\bar{z}\\
			& = \sum_{n,m=1}^{\infty}nm\big\{\hat{\xi}(n)\hat{\psi}(-m)-\hat{\xi}(-m)\hat{\psi}(n)\big\}\int_{\mathbb{D}_{_R}}~z^{n-1}\bar{z}^{m-1}~dz\wedge d\bar{z},
		\end{split}
	\end{equation}
	where we have used Fubini’s theorem to interchange the
	summation and integration because
	\begin{equation*}
		\begin{split}
			& \sum_{n,m=1}^{\infty}\big|nm\big\{\hat{\xi}(n)\hat{\psi}(-m)-\hat{\xi}(-m)\hat{\psi}(n)\big\}\big|\left|\int_{\mathbb{D}_{_R}}~z^{n-1}\bar{z}^{m-1}~dz\wedge d\bar{z}\right|\\
			& \leq 4\pi \sum_{n,m=1}^{\infty} nm\big\{|\hat{\xi}(n)||\hat{\psi}(-m)|+|\hat{\xi}(-m)||\hat{\psi}(n)|\big\}\int_{0}^R~r^{n+m-2}~r~dr\\
			&  \leq 4\pi ~\|\xi\|_{L^1([0,2\pi])}\left(\sum_{n=-\infty}^{\infty}|n\hat{\psi}(n)|\right)\sum_{n,m=1}^{\infty} \frac{nm}{n+m}~R^{n+m}<\infty.
		\end{split}
	\end{equation*}
	Finally from \eqref{eq22} we get 
	\begin{equation*}
		\begin{split}
			& \int_{\mathbb{D}_{_R}} J\big(\widetilde{\xi},\widetilde{\psi}\big)~dz\wedge d\bar{z} 
			= -2i\sum_{n,m=1}^{\infty}nm\big\{\hat{\xi}(n)\hat{\psi}(-m)-\hat{\xi}(-m)\hat{\psi}(n)\big\}
			\times \int_0^R\int_0^{2\pi} r^{n+m-2}~\textup{e}^{i(n-m)t}~rdr dt\\
			& = -4\pi i\sum_{n,m=1}^{\infty}nm\big\{\hat{\xi}(n)\hat{\psi}(-m)-\hat{\xi}(-m)\hat{\psi}(n)\big\}~\left(\frac{R^{n+m}}{n+m}\right)~\delta_{nm}
			= 2\pi i \sum_{n=-\infty}^{\infty} ~n~\hat{\psi}(n)~\hat{\xi}(-n)~R^{2n}
		\end{split}
	\end{equation*}
	and hence 
	\begin{equation}\label{eq23}
		\begin{split}
			\lim_{R\uparrow 1}\int_{\mathbb{D}_{_R}} J\big(\widetilde{\xi},\widetilde{\psi}\big)~dz\wedge d\bar{z} = 2\pi i \sum_{n=-\infty}^{\infty} ~n~\hat{\psi}(n)~\hat{\xi}(-n)
		\end{split}
	\end{equation}
	since $\sum\limits_{n=-\infty}^{\infty}|n\hat{\psi}(n)|<\infty$ and $|\hat{\xi}(n)|\leq \|\xi\|_{L^1([0,2\pi])}<\infty$ for all $n\in \mathbb{Z}$. 
	Thus the conclusion of the theorem follows by combining equations \eqref{eq21} and \eqref{eq23}. This completes the proof. 
\end{proof}

\section{Acknowledgements}
The first author (AC) acknowledges the support  from the Mathematical Research Impact
Centric Support (SERB) project by the  Department of Science $\&$ Technology (DST), G.O.I, and the second author (KBS) thanks Indian National Science Academy for its support through the Senior Scientist Scheme.

\end{document}